\documentclass[preprint,12pt]{elsarticle}
\usepackage{paralist, mathdots} 
\usepackage{amssymb} 
\usepackage{amsmath,tikz,wasysym,multirow} 

\usepackage[hyperindex,breaklinks]{hyperref}

\newtheorem{theorem}{Theorem}[section]
\newtheorem{example}{Example}[section]

\newtheorem{lemma}{Lemma}[section]
\newenvironment{proof}{\textbf{Proof.}}{\qquad $\Box$ \bigskip }

\newcommand{\npmatrix}[1]{\left( \begin{matrix} #1 \end{matrix} \right)}

\begin{document}
\begin{frontmatter}
\title{Integer completely positive matrices of order two}

\author[TJL]{Thomas Laffey}
\ead{thomas.laffey@ucd.ie}

\author[he]{Helena \v Smigoc\corref{cor1}}
\ead{helena.smigoc@ucd.ie}
\address[TJL, he]{School of Mathematics and Statistics, University College Dublin, Belfield, Dublin 4, Ireland}
\cortext[cor1]{Corresponding author}

\begin{abstract}
We show that every integer doubly nonnegative $2 \times 2$ matrix has an integer cp-factorization. 
\end{abstract}


\end{frontmatter}

\section{Introduction}

A $n\times n$ matrix $A$ is said to be \emph{completely positive}, if there exists a (not necessarily square) nonnegative matrix $V$ such that $A=VV^T$. Completely positive matrices of order $n$ form a cone, with the dual cone of copositive matrices, i.e. the matrices with $x^TAx \geq 0$ for all $x$ with nonnegative elements \cite{MR0147484}. Completely positive and copositive matrices have been widely studied, and they play an important role in various applications. However, several basic questions about them are still open. For background, we refer the reader to the following works and citations therein \cite{MR1986666,MR3414584,MR2892529, MR2845851}.

Clearly, any completely positive matrix is nonnegative and positive semidefinite. We call the family of matrices that are both nonnegative and positive semidefinite \emph{doubly nonnegative}. Doubly nonnegative matrices of
order less than $5$ are completely positive \cite{MR0174570}. However, this is no longer true for matrices of oder larger than or equal to $6$ \cite{MR0111694}.

Any $n \times n$ completely positive matrix $A$ has many \emph{cp-factorizations} of the form $A=VV^T,$ where $V$ is an $n \times m$ matrix. Note that $m$ is also not unique. We define the \emph{cp-rank} of $A$ to be the minimal possible $m$. If we demand that $V$ has rational entries, then we say that $A$ has \emph{a rational cp-factorization}. We define the \emph{rational cp-rank} correspondingly. In this note we will study \emph{integer cp-factorizations}, where we demand $V$ to be an integer nonnegative matrix. 

The question, if any rational cp-matrix has a rational cp-factorization is open. However, it is known that every rational matrix which lies in the interior of the cone of completely positive matrices has a rational cp-factorization \cite{MR3624664}. On the other hand, not every $n\times n$ integer completely positive matrix has an integer cp-factorization. In this note we answer a question posed in \cite{Berman2}, by proving that for $n=2$ every integer doubly nonnegative matrix has an integer cp-factorization.


\section{Main Result}
  
First we consider two basic cases: rank one matrices, and matrices with one of the diagonal elements equal to $1$. The main result will be proved by an inductive argument, using those two lemmas as the base of induction. 

\begin{lemma}\label{lem:rank1}
An integer doubly nonnegative matrix of rank $1$ has an integer cp-factorization, and an integer cp-rank less than or equal to $4$. 
\end{lemma}

\begin{proof}
Let $$A=\npmatrix{a & b \\ b& c},$$
where $a,b,c$ are nonnegative integers and $ac=b^2$. In particular, when $b=0$, $A$ has only one nonzero element on the diagonal. In this case the claim follows from Lagrange's four-square theorem, that states that every natural number can be represented as the sum of four squares of integers. 

From now on we assume that $b\neq 0$. Let $d$ denote the greatest common divisor of $a, b, c$: $$a=da_0,\,  b=db_0, c=dc_0,$$ where the greatest common divisor of $a_0, b_0$ and $c_0$ is equal to $1$. Since $A$ has rank $1$, we have $a_0c_0=b_0^2$. Now, if $d'$ divides both $a_0$ and $c_0$, then it has to divide $b_0$ as well. Hence the greatest common divisor of $a_0$ and $c_0$ is $1$. We deduce that $a_0=a_1^2$, $c_0=c_1^2$ and $b_0=a_1c_1$ for some nonnegative integers $a_1$ and $c_1$. Finally, we use Lagrange's four-square theorem to write $d=d_1^2+d_2^2+d_3^2+d_4^2$. Now  
$$ A=d\npmatrix{a_1^2 & a_1c_1 \\ a_1c_1 & c_1^2}
=\sum_{i=1}^4 \npmatrix{d_i a_1 \\ d_i c_1} \npmatrix{d_i a_1 & d_i c_1},$$
illustrates an integer cp-factorization of $A$.  
\end{proof}

\begin{lemma}\label{lem:diag1}
An integer doubly nonnegative matrix with a diagonal element equal to $1$ has an integer cp-factorization. 
\end{lemma}

\begin{proof}
The case when $A$ has rank one has already been dealt with, so we assume that
 $$A=\npmatrix{1 & b \\ b & c}$$
is an integer nonnegative matrix with $c>b^2.$ We can write:
 $$A=\npmatrix{1 \\ b}\npmatrix{1 & b}+\npmatrix{0 & 0 \\ 0 & c-b^2},$$
 and the proof is complete by Lemma \ref{lem:rank1}.  
\end{proof}

\begin{theorem}
An integer doubly nonnegative $2 \times 2$ matrix has an integer cp-factorization.  
\end{theorem}

\begin{proof}
As the case when $A$ has rank $1$ is dealt with in Lemma \ref{lem:rank1}, we may assume that
$$A=\npmatrix{a & b \\ b& c},$$
where $a,b,c$ are positive integers, and $ac>b^2$.

Lemma \ref{lem:rank1} and Lemma \ref{lem:diag1} allow us to use induction, for example on $a+b+c$. Hence, we need to show that, unless $A$ has rank one, there exists a rank one matrix $$R=\npmatrix{k \\ t}\npmatrix{k & t},$$
with $k,t$ nonnegative integers, such that $A-R$ is doubly nonnegative. If such a matrix $R$ exists we will say that $A$ \emph{can be reduced}. Note that $A$ can be reduced, if there exist nonnegative integers $k$ and $t$ such that $k^2\leq a$, $t^2 \leq c$, $kt \leq b$ and $\det(A-R)\geq 0$.

If $$A_1=A-\npmatrix{1 \\ 0 } \npmatrix{1 & 0}$$ is doubly nonnegative, or equivalently, if $(a-1)c\geq b^2$, then the problem reduces to the cp-factorisation of $A_1$. This allows us to assume $(a-1)c <b^2$, and, by a similar argument,  $a(c-1)<b^2$. In particular, we may assume without loss of generalisation that $a\leq b \leq c$.

Let us write $$b=q a+r\text{ and }r^2=\alpha a+ \gamma,$$ where $r, \gamma \in \{0,1,\ldots, a-1\}$, so:
\begin{equation}\label{eq:b^2}
b^2=a(q^2a+2qr+\alpha)+\gamma.
\end{equation} 
We claim that $A$ can be reduced, if 
 $$B=\npmatrix{a & r \\ r & \alpha +1}$$ 
 can be reduced. More precisely, assuming that $B-R'$ is doubly nonnegative for some integer rank one nonnegative matrix $R'$:
 $$R'=\npmatrix{k \\ t'}\npmatrix{k & t'}.$$
 we will prove that $A-R$ is doubly nonnegative for
 $$R=\npmatrix{k \\ qk+t'}\npmatrix{k & qk+t'}.$$ Proving this claim finishes the proof of the theorem, since $B$ has all but one element smaller than $A$, and we have shown earlier, that the statement is true, if one of the diagonal elements is equal to $1$.

Inequalities $ac>b^2>(c-1)a$ together with (\ref{eq:b^2}) imply 
$$q^2a+2qr+\alpha+\frac{\gamma}{a}+1>c>q^2a+2qr+\alpha+\frac{\gamma}{a},$$
and, since $c$ is a positive integer, this gives us 
$$c= q^2a+2qr+\alpha+1.$$
Now the determinant of $A$ is equal to: 
 \begin{align}\label{ineq:det}
 ac-b^2 = a(q^2a+2qr+\alpha+1)-(qa+r)^2=a(\alpha+1)-r^2,
 \end{align}
 and hence equal to the determinant of $B$.

Assuming that $B-R'$ is nonnegative, $A-R$ is nonnegative since  
\begin{align}\label{eq:bkt}
b -k(qk+t')&=aq+r-k(qk+t') \\
                &=q(a-k^2)+(r-k t') \geq 0
\end{align}
and 
\begin{align}\label{eq:ckt}
c-(qk-t')^2&=q^2a+2qr+\alpha+1-(qk-t')^2 \\
  &=q^2(a-k^2)+2q(r-kt')+(\alpha+1-t'^2) \geq 0.
\end{align}
Finally, using (\ref{eq:bkt}) and (\ref{eq:ckt}) we compute
\begin{align*}
\det(A-R)&=(a-k^2)(q^2(a-k^2)+2q(r-kt')+(\alpha+1-t'^2))-(q(a-k^2)+(r-k t'))^2 \\
               &=(a-k^2)(\alpha +1-t'^2)-(r-kt')^2=\det(B-R'), 
\end{align*} 
proving that $\det(A-R)=\det(B-R')\geq 0$.
\end{proof}

The proof of the theorem gives us an algorithmic way to find an integer cp-factorization for a $2 \times 2$ doubly nonnegative matrix. We follow this algorithm on a concrete example below.  

\begin{example}
Let us consider the matrix: $$A_0=\npmatrix{78 & 200 \\ 200 & 4000}.$$ First we reduce the larger diagonal element of $A_0$ as much as we can while still preserving a positive determinant:
$A_0=A_1+D_1,$ where  
$$A_1=\npmatrix{78 & 200 \\ 200 & 513} \text{ and }D_1=\npmatrix{0 & 0 \\ 0 & 3487}.$$
The matrix $D_1$ is covered by Lemma \ref{lem:rank1}. We continue working with $A_1$. 
We write $200=2 \cdot 78+44$, and $44^2=24\cdot 78+64,$ so the matrix corresponding to the matrix $B$ from the proof is equal to 
 $$A_2=\npmatrix{78 & 44 \\ 44 & 25}.$$
 Note that we cannot decrease $78$ on the diagonal while preserving a positive determinant. Now $25$ is the smaller element on the diagonal. We write $44=25+19$ and $19^2=14\cdot 25+11$. Following the proof, the problem reduces to the matrix
  $$A_3=\npmatrix{15 & 19 \\ 19 & 25 }.$$
 The diagonal element $25$ cannot be reduced, so we write $19=15+4$ and $4^2=15+1.$ This gives us 
  $$A_4=\npmatrix{15 & 4 \\ 4 & 2}.$$
  At this point $15$ can be reduced to $8$:
  $$A_4=\npmatrix{8 & 4 \\ 4 & 2}+\npmatrix{7 & 0 \\ 0 & 0}.$$
  The matrix that we are left with has rank one:
\begin{align*}
A_5&=\npmatrix{8 & 4 \\ 4 & 2}=2\npmatrix{4 & 2 \\ 2 & 1} \\
&=\npmatrix{2 \\ 1 }\npmatrix{2 & 1}+\npmatrix{2 \\ 1 }\npmatrix{2 & 1}. 
\end{align*}
 We take 
 $$R_4=R_5=\npmatrix{2 \\ 1 }\npmatrix{2 & 1},$$
noting that $A_5-R_5$ and $A_4-R_4$ are doubly nonnegative. Following the proof of the theorem, we take $k=2$, $t'=1$, and $q=1$ to construct:
 $$R_3=\npmatrix{2 \\ 3 }\npmatrix{2 & 3},$$
 with $A_3-R_3$ doubly nonnegative. To get $R_2$, we first note that the smaller diagonal element of $A_2$ is in $(2,2)$ position. We take $k=3,$ $t'=2$ and $q=1$ to get
  $$R_2=\npmatrix{5 \\ 3 }\npmatrix{5 & 3}$$ 
  with $A_2-R_2$ doubly nonnegative. Finally, to get $R_1$ we take $k=5$, $t'=3$ and $q=2$:
   $$R_1=\npmatrix{5 \\ 13 }\npmatrix{5 & 13}.$$
  Now we need to repeat the process for  $A_1-R_1,$ which is by construction doubly nonnegative. 

If we follow the algorithm to completion, we get the following decomposition of $A_0$ that contains $10$ terms:
\begin{align*}
A_0&=\left(
\begin{array}{c}
 0 \\
 59 \\
\end{array}
\right)\left(
\begin{array}{cc}
 0 & 59 \\
\end{array}
\right)+\left(
\begin{array}{c}
 0 \\
 2 \\
\end{array}
\right)\left(
\begin{array}{cc}
 0 & 2 \\
\end{array}
\right)+\left(
\begin{array}{c}
 0 \\
 1 \\
\end{array}
\right)\left(
\begin{array}{cc}
 0 & 1 \\
\end{array}
\right)+\left(
\begin{array}{c}
 0 \\
 1 \\
\end{array}
\right)\left(
\begin{array}{cc}
 0 & 1 \\
\end{array}
\right) \\
&+\left(
\begin{array}{c}
 2 \\
 5 \\
\end{array}
\right)\left(
\begin{array}{cc}
 2 & 5 \\
\end{array}
\right)+\left(
\begin{array}{c}
 2 \\
 5 \\
\end{array}
\right)\left(
\begin{array}{cc}
 2 & 5 \\
\end{array}
\right)+\left(
\begin{array}{c}
 2 \\
 5 \\
\end{array}
\right)\left(
\begin{array}{cc}
 2 & 5 \\
\end{array}
\right)\\
&+\left(
\begin{array}{c}
 4 \\
 10 \\
\end{array}
\right)\left(
\begin{array}{cc}
 4 & 10 \\
\end{array}
\right)+\left(
\begin{array}{c}
 5 \\
 13 \\
\end{array}
\right)\left(
\begin{array}{cc}
 5 & 13 \\
\end{array}
\right)+\left(
\begin{array}{c}
 5 \\
 13 \\
\end{array}
\right)\left(
\begin{array}{cc}
 5 & 13 \\
\end{array}
\right)
\end{align*}

On the other hand, an ad hoc decomposition with only $8$ terms can also be obtained: 
\begin{align*}
A_0&=\left(
\begin{array}{c}
 8 \\
 25 \\
\end{array}
\right)\left(
\begin{array}{cc}
 8 & 25 \\
\end{array}
\right)+\left(
\begin{array}{c}
 0 \\
 58 \\
\end{array}
\right)\left(
\begin{array}{cc}
 0 & 58 \\
\end{array}
\right)+\left(
\begin{array}{c}
 0 \\
 3 \\
\end{array}
\right)\left(
\begin{array}{cc}
 0 & 3 \\
\end{array}
\right)+\left(
\begin{array}{c}
 0 \\
 1 \\
\end{array}
\right)\left(
\begin{array}{cc}
 0 & 1 \\
\end{array}
\right)\\ 
&+\left(
\begin{array}{c}
 0 \\
 1 \\
\end{array}
\right)\left(
\begin{array}{cc}
 0 & 1 \\
\end{array}
\right)+\left(
\begin{array}{c}
 3 \\
 0 \\
\end{array}
\right)\left(
\begin{array}{cc}
 3 & 0 \\
\end{array}
\right)+\left(
\begin{array}{c}
 2 \\
 0 \\
\end{array}
\right)\left(
\begin{array}{cc}
 2 & 0 \\
\end{array}
\right)+\left(
\begin{array}{c}
 1 \\
 0 \\
\end{array}
\right)\left(
\begin{array}{cc}
 1 & 0 \\
\end{array}
\right)
\end{align*}
This shows that the algorithm we presented does not produce a decomposition with the smallest possible cp-rank. 
 \end{example}


\end{document}